\documentclass[a4paper,12pt]{article}

\usepackage[utf8]{inputenc}
\usepackage[english]{babel}

\usepackage{amsmath,amssymb,amsfonts,amsthm,mathtools}
\usepackage{mathrsfs}

\usepackage{geometry}
\newtheorem{theorem}{Theorem}[section]
\newtheorem{lemma}[theorem]{Lemma}
\newtheorem{proposition}[theorem]{Proposition}

\theoremstyle{definition}
\newtheorem{definition}[theorem]{Definition}

\theoremstyle{remark}
\newtheorem{remark}[theorem]{Remark}

\theoremstyle{definition}
\newtheorem{example}[theorem]{Example}

\newcommand{\Lp}{L^p[0,1]}

\newcommand{\F}{\mathcal{F}_W}

\newcommand{\Nzero}{\mathbb{N}_0}

\usepackage[unicode]{hyperref}
\hypersetup{
  colorlinks=true,
  linkcolor=blue,
  citecolor=blue,
  urlcolor=blue
}

\newcommand{\DOI}[1]{\href{https://doi.org/#1}{\nolinkurl{doi:#1}}}

\title{\textbf{Compactness and Spectral Properties of Multiplier Operators in the Walsh System}}

\author{M. Ruzhansky\thanks{The authors are supported by the FWO Odysseus 1 grant G.0H94.18N: Analysis and Partial Differential Equations and by the Methusalem programme of the Ghent University Special Research Fund (BOF) (Grant number 01M01021). Michael Ruzhansky is supported by the FWO Senior Research Grant G011522N and by the EPSRC grant UKRI3645.}
 \and S. Episkoposian\footnotemark[1] \and R. Yeghoyan\footnotemark[1]}

\date{}

\begin{document}

\maketitle

\begin{abstract}
We investigate compactness and spectral properties of multiplier operators associated with the Walsh system in the spaces $\Lp$, $1<p<\infty$.
Building upon previously established criteria for boundedness of Walsh multipliers, we prove an exact compactness criterion in the $L^p\to L^p$ regime for all $1<p<\infty$(assuming boundedness of the multiplier), and also in the $L^p\to L^2$ regime for $2<p<\infty$.
The key result states that compactness is equivalent to the condition $a_n\to 0$ for the multiplier symbol.
We also examine in detail the point spectrum and derive strict spectral inclusions; in the Hilbert space case $p=2$ we obtain a complete description of the spectrum.
For $p\neq 2$, we emphasize the limitations of transferring "diagonal" arguments and formulate results in a form that does not admit incorrect generalizations.
\end{abstract}



\section{Introduction}
\label{sec:intro}

The Walsh function system, first introduced by J.~L.~Walsh~\cite{Walsh1923},
represents one of the fundamental orthonormal systems
based on the dyadic structure of the interval $[0,1]$.
Unlike the trigonometric system,
Walsh functions take only values $\pm1$ almost everywhere,
a fact that was thoroughly studied already in early works by R.~Paley~\cite{Paley1932}
and N.~Fine~\cite{Fine1949}.
Precisely this discrete-dyadic nature,
connected with the binary structure of the argument,
substantially influences the analytic and operator properties
arising in expansions with respect to the Walsh system,
and leads to noticeable differences from the classical trigonometric case.

The functional-analytic theory of Walsh series and transforms
was systematically developed in the monograph
by B.~Golubov, A.~Efimov, and V.~Skvortsov~\cite{GolubovEfimovSkvortsov1991},
as well as in the book by Schipp, Wade, and Simon~\cite{SchippWadeSimon1990}.
In particular, it has been established that the Walsh system forms
a Schauder basis in the spaces $L^p[0,1]$ for $1<p<\infty$;
however, unlike the case $p=2$,
it is not an unconditional basis for $p\neq2$.
This fact, well known in the theory of Banach spaces
(see, e.g.,~\cite{Wojtaszczyk}),
plays a fundamental role in the investigation of operators
that are diagonal with respect to this basis.
Moreover, for Walsh–Fourier coefficients
analogues of the Hausdorff–Young inequality hold,
and the synthesis operator is well-defined in $L^p$,
which makes the Walsh system especially convenient
for studying multiplier operators.

Multiplier operators,
arising in the context of orthonormal expansions,
are classical objects of harmonic analysis.
For trigonometric series,
fundamental results by Marcinkiewicz~\cite{Marcinkiewicz1939},
Hirschman~\cite{Hirschman1961},
and Zygmund~\cite{Zygmund1959}
have shown that boundedness, compactness,
and spectral properties of multipliers
in $L^p$ spaces possess a subtle
and often counterintuitive structure.
The general Banach-space approach
to the theory of $L^p$-multipliers
was formulated in the work of de~Leeuw~\cite{deLeeuw1965}
and subsequently developed within the framework of abstract operator theory.

Historically, the investigation of operators
diagonal with respect to orthonormal systems
goes back to the classical works of F.~Riesz~\cite{Riesz1910},
where the foundations of spectral analysis of linear operators
in functional spaces were laid,
and to the Tauberian theory of N.~Wiener~\cite{Wiener1932},
which showed that asymptotic properties of expansion coefficients
substantially influence the invertibility and spectrum of the corresponding operators.
These ideas later became the foundation
for the development of multiplier and diagonal operator theory
in $L^p$ spaces.

From the functional-analytic viewpoint,
multiplier operators in the Walsh system
are naturally interpreted as diagonal operators
with respect to a fixed Schauder basis.
General properties of such operators,
including compactness criteria,
are closely related to the behavior of their diagonal coefficients
and have been thoroughly studied in the theory of Banach spaces
(see the classical monograph by Lindenstrauss and Tzafriri~\cite{LT}).
However, the concrete realization of these abstract principles
substantially depends on the structure of the basis,
which makes the Walsh system a special and nontrivial example.

In a recent work by the authors~\cite{BoundednessWalsh},
strict conditions for boundedness
of multiplier operators in the Walsh system
in the spaces $L^p[0,1]$ were obtained.
The present article naturally continues this investigation,
focusing on compactness and spectral properties of such operators.
We emphasize that for $p\neq2$
operators in $L^p$ spaces
act in a non-Hilbert environment,
and therefore many spectral statements
that are obvious in $L^2$
do not automatically transfer
to Banach spaces
(see, e.g.,~\cite{SteinHA,Conway,Fefferman1971}).

The main goal of the present work is to obtain
strict and fully correct compactness criteria
for multiplier operators in the Walsh system,
as well as a careful analysis of their spectrum
in the form of guaranteed inclusions.
At the same time, we consciously avoid
using implicit Hilbert-space arguments,
which allows us to clearly distinguish
the Banach-space nature
of the considered operator effects.

\section{Preliminaries and notation}
\label{sec:prelim}

\subsection{Walsh system and Walsh–Fourier coefficients}

\begin{definition}[Rademacher and Walsh--Paley functions]
For $k\in\mathbb{N}_0$ define the Rademacher functions on $[0,1)$ by
\[
r_k(x):=\operatorname{sign}\bigl(\sin(2^{k+1}\pi x)\bigr),
\qquad x\in[0,1).
\]
If $n=\sum_{k=0}^m b_k 2^k$ is the dyadic expansion of $n$ with $b_k\in\{0,1\}$,
then the Walsh--Paley functions are defined by
\[
W_n(x):=\prod_{k=0}^m r_k(x)^{\,b_k},\qquad x\in[0,1).
\]
\end{definition}
\noindent
Each $W_n$ takes only the values $\pm1$ a.e. and $\{W_n\}_{n\ge0}$ is an orthonormal system in $L^2[0,1]$.

The Walsh functions satisfy
\[
\int_0^1 W_n(x)W_m(x)\,dx = \delta_{nm}.
\]

For $f\in L^1[0,1]$ the Walsh–Fourier coefficients are defined by the formula
\[
\hat{f}(n) := \int_0^1 f(x)W_n(x)\,dx,\qquad n\in\Nzero.
\]
We also use the notation $\F f := (\hat{f}(n))_{n\ge 0}$.

It is known (see~\cite{GolubovEfimovSkvortsov1991,SchippWadeSimon1990}) that
for $1<p<\infty$ the Walsh system forms a Schauder basis in $\Lp$,
and hence for each $f\in\Lp$ the Walsh–Fourier series converges to $f$ in the $\Lp$ norm.

\subsection{Multiplier operators}

\begin{definition}[Walsh multiplier]
Let $a=\{a_n\}_{n\ge0}$ be a fixed sequence of complex numbers.
The \emph{multiplier operator in the Walsh system} is the operator
formally defined by the equality
\[
T_a f = \sum_{n=0}^{\infty} a_n\, \hat{f}(n)\, W_n,
\]
on its natural domain of definition.
\end{definition}

Depending on the purpose, we will be interested in boundedness regimes:
\[
T_a:\Lp\to\Lp \quad\text{or}\quad T_a: L^p[0,1]\to L^2[0,1].
\]
Criteria for boundedness in the respective regimes are collected in~\cite{BoundednessWalsh}.
In this article we assume that $T_a$ is already defined as a bounded operator
in the considered regime, and we investigate its compactness and spectrum.

\subsection{Analysis and synthesis inequalities (Walsh)}
\label{subsec:HYsynthesis}

We fix two standard estimates.

\begin{lemma}[Hausdorff–Young for Walsh]
\label{lem:HY}
Let $1<p\le 2$ and $p'=\frac{p}{p-1}$.
Then there exists a constant $C_{HY}(p)>0$ such that for all $f\in \Lp$
\[
\|\F f\|_{\ell^{p'}} \le C_{HY}(p)\,\|f\|_{L^p}.
\]
\end{lemma}

\begin{proof}
This is a special case of the Hausdorff–Young inequality on a compact abelian group,
applied to the dyadic group and its character represented by the Walsh system;
see, e.g.,~\cite[Ch.~II]{GolubovEfimovSkvortsov1991} or~\cite[Ch.~2]{SchippWadeSimon1990}.
\end{proof}

\begin{lemma}[Synthesis in $\Lp$ for Walsh]
\label{lem:synthesis}
Let $1<p<2$ and $p'=\frac{p}{p-1}$.
Then there exists a constant $C_{S}(p)>0$ such that for any sequence
$c=\{c_n\}\in\ell^{p'}$ the series $\sum_{n\ge0} c_nW_n$ converges in $\Lp$ and
\[
\left\|\sum_{n=0}^{\infty} c_nW_n\right\|_{L^p}\le C_S(p)\,\|c\|_{\ell^{p'}}.
\]
\end{lemma}

\begin{proof}
This statement is a known fact of Walsh series theory;
see \cite[Ch.~II, section on Walsh transform and synthesis]{GolubovEfimovSkvortsov1991}
and also \cite{SchippWadeSimon1990}.
\end{proof}

\begin{remark}
In what follows it is important to keep track of which constants appear in the estimates.
We will explicitly indicate the dependence of the final constants on $C_{HY}(p)$ and $C_S(p)$,
as well as on the norm $\|T_a\|$ as a bounded operator (when required).
\end{remark}

\subsection{Bounded symbols and boundedness of multipliers}
\label{subsec:linfty}

\begin{lemma}
\label{lem:linftyMultiplier}
Let $1<p<\infty$ and $b=\{b_n\}_{n\ge0}\in \ell^\infty$.
Then the multiplier $T_b$ is well-defined on the dense subspace
of finite linear combinations of Walsh functions and extends by continuity
to a bounded operator $T_b:\Lp\to\Lp$.
Moreover:
\begin{enumerate}
\item[(a)] if $1<p<2$, then
\[
\|T_b\|_{L^p\to L^p}\le C_S(p)\,C_{HY}(p)\,\|b\|_{\ell^\infty};
\]
\item[(b)] if $2<p<\infty$, then
\[
\|T_b\|_{L^p\to L^p}\le C_S(p')\,C_{HY}(p')\,\|b\|_{\ell^\infty},
\quad p'=\frac{p}{p-1}\in(1,2).
\]
\end{enumerate}
\end{lemma}

\begin{proof}
We first define $T_b$ on the dense subspace
\[
\mathcal P:=\mathrm{span}\{W_n:\,n\ge0\}
\]
of finite Walsh polynomials by
\[
T_b\Big(\sum_{n=0}^N c_n W_n\Big):=\sum_{n=0}^N b_n c_n W_n.
\]

\smallskip
\textbf{Case $1<p<2$.}
Let $f\in\mathcal P$ and write $f=\sum_{n=0}^N \hat f(n)\,W_n$.
Then
\[
T_b f=\sum_{n=0}^N b_n\,\hat f(n)\,W_n.
\]
By the synthesis estimate of Lemma~\ref{lem:synthesis} and the bound
$\|(b_n\hat f(n))\|_{\ell^{p'}}\le \|b\|_{\ell^\infty}\,\|\F f\|_{\ell^{p'}}$ we get
\[
\|T_b f\|_{L^p}
\le C_S(p)\,\|(b_n\hat f(n))\|_{\ell^{p'}}
\le C_S(p)\,\|b\|_{\ell^\infty}\,\|\F f\|_{\ell^{p'}}.
\]
Applying Hausdorff--Young (Lemma~\ref{lem:HY}) yields
\[
\|T_b f\|_{L^p}\le C_S(p)\,C_{HY}(p)\,\|b\|_{\ell^\infty}\,\|f\|_{L^p},
\qquad f\in\mathcal P.
\]
Since $\mathcal P$ is dense in $L^p[0,1]$, this estimate shows that $T_b$
extends uniquely by continuity to a bounded operator on $L^p[0,1]$, and
\[
\|T_b\|_{L^p\to L^p}\le C_S(p)\,C_{HY}(p)\,\|b\|_{\ell^\infty}.
\]

\smallskip
\textbf{Case $2<p<\infty$.}
Let $p'=\frac{p}{p-1}\in(1,2)$.

\emph{Step 1: boundedness on $L^{p'}$.}
By the already proved case $1<p'<2$ applied to $\overline b$, the operator
$T_{\overline b}$ extends uniquely to a bounded operator on $L^{p'}[0,1]$ and
\[
\|T_{\overline b}\|_{L^{p'}\to L^{p'}}
\le C_S(p')\,C_{HY}(p')\,\|b\|_{\ell^\infty}.
\]

\emph{Step 2: definition and boundedness on $L^{p}$ by duality.}
Define $T_b:L^p[0,1]\to L^p[0,1]$ as the Banach-space adjoint of
$T_{\overline b}$, i.e.
\[
\langle T_b f, g\rangle := \langle f, T_{\overline b} g\rangle,
\qquad f\in L^p,\ g\in L^{p'}.
\]
Then $T_b$ is bounded and
\[
\|T_b\|_{L^p\to L^p}=\|T_{\overline b}\|_{L^{p'}\to L^{p'}}.
\]

\emph{Step 3: consistency with the multiplier on $\mathcal P$.}
For $f,g\in\mathcal P$, using orthogonality of $\{W_n\}$,
\[
\langle T_b f, g\rangle
=\sum_{n\ge0} b_n\,\hat f(n)\,\overline{\hat g(n)},
\]
so on $\mathcal P$ this operator coincides with the formal Walsh multiplier
$\sum_{n\ge0} b_n\hat f(n)W_n$. This proves item \textup{(b)}.

\end{proof}

\subsection{Norm of the difference of Walsh functions}
\label{subsec:normdiff}

The following lemma strengthens the necessity proof in the compactness criterion.

\begin{lemma}
\label{lem:WalshDistance}
Let $1\le p<\infty$ and $n\neq m$.
Then
\[
\|W_n-W_m\|_{L^p} = 2^{\,1-\frac1p}.
\]
\end{lemma}

\begin{proof}
Since $W_n,W_m$ take values $\pm 1$ almost everywhere,
$W_n-W_m$ takes values $0$ or $\pm 2$ almost everywhere.
Because $W_nW_m = W_{n\oplus m}$ (operation $\oplus$ is bitwise XOR of indices),
$W_nW_m$ is a nontrivial Walsh function for $n\neq m$ and has zero mean:
\[
\int_0^1 W_n(x)W_m(x)\,dx = 0.
\]
From the $\{\pm1\}$-valuedness it follows that the sets
\[
E:=\{x\in[0,1]: W_n(x)=W_m(x)\}=\{W_nW_m=1\},
\]
\[
F:=\{x\in[0,1]: W_n(x)\neq W_m(x)\}=\{W_nW_m=-1\}
\]
have equal measures: $|E|=|F|=\tfrac12$.
On $E$ we have $W_n-W_m=0$, while on $F$ we have $|W_n-W_m|=2$.
Consequently,
\[
\|W_n-W_m\|_{L^p}^p
=
\int_0^1 |W_n-W_m|^p\,dx
=
\int_F 2^p\,dx
=
2^p\cdot \frac12
=
2^{p-1}.
\]
Taking the $p$th root, we obtain
$\|W_n-W_m\|_{L^p}=2^{1-\frac1p}$.
\end{proof}
\begin{remark}
The statement of Lemma~\ref{lem:WalshDistance}
extends to $p=\infty$. Indeed, since $W_n,W_m\in\{\pm1\}$ a.e.,
the difference $W_n-W_m$ takes only values $0$ and $\pm2$ a.e., and for $n\ne m$
the set $\{x: W_n(x)\ne W_m(x)\}$ has measure $1/2$. Hence
\[
\|W_n-W_m\|_{L^\infty}=2,
\]
in agreement with $2^{\,1-1/p}\to2$ as $p\to\infty$.
\end{remark}

\section{Compactness of multiplier operators}
\label{sec:compact}

\subsection{Compactness criterion in the $L^p\to L^p$ regime for $1<p<2$}

\begin{theorem}
\label{thm:compactLp}
Let $1<p<\infty$ and $T_a:\Lp\to\Lp$ be a bounded Walsh multiplier operator.
Then the following conditions are equivalent:
\begin{enumerate}
\item[(i)] $T_a$ is compact in $\Lp$;
\item[(ii)] $a_n\to 0$ as $n\to\infty$.
\end{enumerate}
\end{theorem}

\begin{proof}
We first prove the equivalence for $1<p<2$.
The case $2<p<\infty$ then follows by duality.

\textbf{(i)$\Rightarrow$(ii).}
Assume the contrary: $a_n\nrightarrow 0$.
Then there exist $\varepsilon>0$ and a sequence of indices $\{n_k\}$ such that
$|a_{n_k}|\ge\varepsilon$ for all $k$.

Set
\[
u_k:=\frac{a_{n_k}}{|a_{n_k}|}\in\mathbb T,\qquad
f_k:=u_k\,W_{n_k}.
\]
Then $\|f_k\|_{L^p}=1$ and
\[
T_af_k = |a_{n_k}|\,W_{n_k}.
\]

For $k\neq j$ compute the norm of the difference. Since $n_k\neq n_j$,
$\int_0^1 W_{n_k}(x)W_{n_j}(x)\,dx=0$, and therefore the sets
\[
E:=\{x\in[0,1]: W_{n_k}(x)=W_{n_j}(x)\},\qquad
F:=\{x\in[0,1]: W_{n_k}(x)\neq W_{n_j}(x)\}
\]
have measures $|E|=|F|=\frac12$.
On $E$ we have
$|a_{n_k}|W_{n_k}-|a_{n_j}|W_{n_j}=\pm\bigl(|a_{n_k}|-|a_{n_j}|\bigr)$,
while on $F$ we have
$|a_{n_k}|W_{n_k}-|a_{n_j}|W_{n_j}=\pm\bigl(|a_{n_k}|+|a_{n_j}|\bigr)$.
Hence,
\[
\|T_af_k-T_af_j\|_{L^p}^p
= \frac12\bigl||a_{n_k}|-|a_{n_j}|\bigr|^p
 + \frac12\bigl(|a_{n_k}|+|a_{n_j}|\bigr)^p.
\]
Since $|a_{n_k}|,|a_{n_j}|\ge\varepsilon$, we have $|a_{n_k}|+|a_{n_j}|\ge 2\varepsilon$, and thus
\[
\|T_af_k-T_af_j\|_{L^p}^p \ge \frac12(2\varepsilon)^p
= 2^{p-1}\varepsilon^p,
\]
i.e.,
\[
\|T_af_k-T_af_j\|_{L^p}\ge 2^{1-1/p}\varepsilon\qquad(k\neq j).
\]
Therefore, the sequence $\{T_af_k\}$ has no convergent subsequence in $L^p$,
which contradicts the compactness of $T_a$. Hence, $a_n\to0$.

\medskip
\textbf{(ii)$\Rightarrow$(i).}
Let $a_n\to0$.
For $N\in\mathbb{N}$ consider the truncations
\[
T_a^{(N)}f := \sum_{n=0}^N a_n\,\hat{f}(n)\,W_n.
\]
Each $T_a^{(N)}$ has finite rank, and consequently is compact in $\Lp$.

We show that $T_a^{(N)}\to T_a$ in the operator norm $\Lp\to\Lp$.
Let $f\in\Lp$. Then
\[
(T_a-T_a^{(N)})f=\sum_{n>N} a_n\,\hat{f}(n)\,W_n.
\]
Set $c_n:=a_n\,\hat{f}(n)$ for $n>N$.
By the Synthesis Lemma~\ref{lem:synthesis} (for $1<p<2$) we have
\[
\|(T_a-T_a^{(N)})f\|_{L^p}
\le C_S(p)\,\|(c_n)_{n>N}\|_{\ell^{p'}}.
\]
Furthermore,
\[
\|(c_n)_{n>N}\|_{\ell^{p'}}
\le \Bigl(\sup_{n>N}|a_n|\Bigr)\,\|(\hat{f}(n))_{n>N}\|_{\ell^{p'}}
\le \Bigl(\sup_{n>N}|a_n|\Bigr)\,\|\F f\|_{\ell^{p'}}.
\]
By the Hausdorff–Young inequality (Lemma~\ref{lem:HY})
\[
\|\F f\|_{\ell^{p'}}\le C_{HY}(p)\,\|f\|_{L^p}.
\]
Thus,
\[
\|(T_a-T_a^{(N)})f\|_{L^p}
\le C_S(p)\,C_{HY}(p)\,\Bigl(\sup_{n>N}|a_n|\Bigr)\,\|f\|_{L^p},
\]
i.e.,
\[
\|T_a-T_a^{(N)}\|_{L^p\to L^p}
\le C_S(p)\,C_{HY}(p)\,\sup_{n>N}|a_n|\xrightarrow[N\to\infty]{}0.
\]
Consequently, $T_a$ is the norm limit of the compact operators $T_a^{(N)}$,
hence compact.

\medskip
\textbf{Duality step ($2<p<\infty$).}
Let $2<p<\infty$ and set $p'=\frac{p}{p-1}\in(1,2)$.
By the Schauder theorem, a bounded operator is compact if and only if its adjoint is compact.

As in the proof of Lemma~\ref{lem:linftyMultiplier}, we have
\[
(T_a)^* = T_{\overline a}\quad \text{on }L^{p'}[0,1].
\]
Hence $T_a$ is compact on $L^p$ if and only if $T_{\overline a}$ is compact on $L^{p'}$.
Applying the already proved case $1<p'<2$ yields $\overline{a_n}\to0$, and therefore $a_n\to0$.
Conversely, if $a_n\to0$, then $\overline{a_n}\to0$ and compactness of $T_{\overline a}$ on $L^{p'}$
implies compactness of $T_a$ on $L^p$ by the same reflexive duality argument.
\end{proof}
\medskip

\begin{remark}
The obtained compactness criterion
for multiplier operators in the Walsh system
naturally agrees with
the general theory of diagonal operators
in Banach spaces.
In particular, it is known that for operators
diagonal with respect to a Schauder basis,
compactness is equivalent to the tendency
of the diagonal coefficients to zero
under the condition of boundedness of the operator
(see, e.g.,~\cite{LT}).
In this sense, the proof of compactness
via approximation by finite-rank operators
is not a specific trick,
but represents an implementation
of a general functional-analytic principle.

At the same time, the concrete form of this criterion
substantially depends on the properties of the basis.
For the Walsh system, the key roles are played by
the dyadic structure,
the $\pm1$-valuedness of the basis functions,
and the well-definedness of the synthesis operator in $L^p$,
which distinguishes this case
both from the trigonometric system
and from other orthonormal expansions.
Thus, the obtained results
should be considered
as a nontrivial concretization
of the general theory of diagonal operators
in the specific and important context
of harmonic analysis with respect to the Walsh system.
\end{remark}

\subsection{Compactness in the $L^p\to L^2$ regime for $2<p<\infty$}

\begin{theorem}
\label{thm:compactLpL2}
Let $2<p<\infty$ and $T_a:L^p[0,1]\to L^2[0,1]$ be a bounded Walsh multiplier.
Then $T_a$ is compact if and only if $a_n\to 0$ as $n\to\infty$.
\end{theorem}

\begin{proof}
First note that for $p>2$ we have the continuous embedding
$L^p[0,1]\hookrightarrow L^2[0,1]$:
\[
\|f\|_{L^2}\le \|f\|_{L^p},
\]
since the measure of $[0,1]$ equals $1$.

\medskip
\textbf{Necessity.}
Assume that $T_a$ is compact but $a_n\nrightarrow 0$.
Then there exist $\varepsilon>0$ and a sequence of indices $\{n_k\}$ such that
$|a_{n_k}|\ge \varepsilon$ for all $k$.
Set $f_k:=W_{n_k}$. Then $\|f_k\|_{L^p}=1$ and
\[
T_af_k=a_{n_k}W_{n_k}.
\]
For $k\neq j$ estimate the distance in $L^2$ using orthonormality:
\[
\|T_af_k-T_af_j\|_{L^2}^2
=\|a_{n_k}W_{n_k}-a_{n_j}W_{n_j}\|_{L^2}^2
=|a_{n_k}|^2+|a_{n_j}|^2
\ge 2\varepsilon^2.
\]
Hence,
\[
\|T_af_k-T_af_j\|_{L^2}\ge \sqrt{2}\,\varepsilon>0
\qquad (k\neq j).
\]

Thus, $\{T_af_k\}$ contains no convergent subsequence in $L^2$,
which contradicts the compactness of $T_a$ as an operator $L^p\to L^2$.
Consequently, $a_n\to 0$.

\medskip
\textbf{Sufficiency.}
Let $a_n\to 0$.
Define the truncations
\[
T_a^{(N)}f:=\sum_{n=0}^N a_n\,\hat{f}(n)\,W_n.
\]
Each $T_a^{(N)}$ has finite rank as an operator $L^p\to L^2$
(its image lies in $\mathrm{span}\{W_0,\dots,W_N\}\subset L^2$),
hence is compact.

Estimate the tail for $f\in L^p$.
By orthonormality of $\{W_n\}$ in $L^2$ and Parseval's identity we obtain
\begin{align*}
\|(T_a-T_a^{(N)})f\|_{L^2}^2
&=\left\|\sum_{n>N} a_n\,\hat{f}(n)\,W_n\right\|_{L^2}^2
=\sum_{n>N}|a_n|^2\,|\hat{f}(n)|^2\\
&\le \left(\sup_{n>N}|a_n|^2\right)\sum_{n>N}|\hat{f}(n)|^2
\le \left(\sup_{n>N}|a_n|^2\right)\sum_{n\ge0}|\hat{f}(n)|^2\\
&=\left(\sup_{n>N}|a_n|^2\right)\|f\|_{L^2}^2
\le \left(\sup_{n>N}|a_n|^2\right)\|f\|_{L^p}^2.
\end{align*}
Therefore,
\[
\|(T_a-T_a^{(N)})f\|_{L^2}\le \Bigl(\sup_{n>N}|a_n|\Bigr)\,\|f\|_{L^p},
\]
i.e.,
\[
\|T_a-T_a^{(N)}\|_{L^p\to L^2}\le \sup_{n>N}|a_n|\xrightarrow[N\to\infty]{}0.
\]
Hence, $T_a$ is the norm limit of the compact operators $T_a^{(N)}$,
and therefore compact.
\end{proof}

\begin{remark}
The proof in the $L^p\to L^2$ regime essentially uses the Hilbert space structure of the target space $L^2$:
in particular, Parseval's equality and diagonality with respect to the orthonormal Walsh basis.
\end{remark}

\section{Spectral properties}
\label{sec:spectral}

\subsection{Point spectrum}

\begin{proposition}
\label{prop:point}
Let $T_a$ be a Walsh multiplier, properly defined as a bounded operator in $\Lp$, $1<p<\infty$.
Then for each $n\in\Nzero$ the number $a_n$ belongs to the point spectrum of $T_a$,
and $W_n$ is an eigenfunction:
\[
T_aW_n=a_nW_n.
\]
\end{proposition}

\begin{proof}
Directly from the definition of the multiplier:
\[
\hat{W_n}(k)=\int_0^1 W_n(x)W_k(x)\,dx=\delta_{nk},
\]
hence
\[
T_aW_n=\sum_{k=0}^\infty a_k\,\hat{W_n}(k)\,W_k=a_nW_n.
\]
\end{proof}

\subsection{Spectrum in the case $p=2$}

In the Hilbert space case $L^2[0,1]$ the Walsh multiplier is a diagonal operator
with respect to the orthonormal basis $\{W_n\}$.

\begin{theorem}[Complete description of the spectrum for $p=2$]
\label{thm:spectrumL2}
Let $T_a:L^2[0,1]\to L^2[0,1]$ be a bounded Walsh multiplier.
Then
\[
\sigma(T_a)=\overline{\{a_n:\,n\ge 0\}}.
\]
Moreover, $\lambda\in\rho(T_a)$ if and only if $\inf_{n\ge0}|a_n-\lambda|>0$,
and in this case
\[
\|(T_a-\lambda I)^{-1}\|_{L^2\to L^2}=\sup_{n\ge0}\frac{1}{|a_n-\lambda|}.
\]
\end{theorem}

\begin{proof}
Let $\lambda\in\mathbb{C}$.
Consider the operator $T_a-\lambda I$, which acts as a multiplier with symbol $(a_n-\lambda)$.
In the basis $\{W_n\}$ it has diagonal form:
\[
(T_a-\lambda I)\left(\sum_{n\ge0} c_nW_n\right)=\sum_{n\ge0}(a_n-\lambda)c_nW_n,
\]
for all vectors from the dense subspace of finite linear combinations.

\smallskip
\textbf{1) If $\inf_n|a_n-\lambda|>0$, then $\lambda\in\rho(T_a)$.}
Formally define
\[
R_\lambda f:=\sum_{n\ge0}\frac{\hat{f}(n)}{a_n-\lambda}\,W_n.
\]
By Parseval,
\[
\|R_\lambda f\|_{L^2}^2=\sum_{n\ge0}\left|\frac{\hat{f}(n)}{a_n-\lambda}\right|^2
\le \left(\sup_{n\ge0}\frac{1}{|a_n-\lambda|^2}\right)\sum_{n\ge0}|\hat{f}(n)|^2
=
\left(\sup_{n\ge0}\frac{1}{|a_n-\lambda|^2}\right)\|f\|_{L^2}^2.
\]
Hence $R_\lambda$ is a bounded operator and is the inverse of $T_a-\lambda I$,
i.e., $\lambda\in\rho(T_a)$, and the resolvent norm equals the indicated supremum estimate.

\smallskip
\textbf{2) If $\inf_n|a_n-\lambda|=0$, then $\lambda\in\sigma(T_a)$.}
Then there exists a sequence $n_k$ such that $|a_{n_k}-\lambda|\to0$.
Consider $f_k:=W_{n_k}$; then $\|f_k\|_{L^2}=1$, and
\[
\|(T_a-\lambda I)f_k\|_{L^2}=\|(a_{n_k}-\lambda)W_{n_k}\|_{L^2}=|a_{n_k}-\lambda|\to0.
\]
Consequently, $T_a-\lambda I$ cannot be boundedly invertible (otherwise from $\|(T_a-\lambda I)f_k\|\to0$
it would follow that $\|f_k\|\to0$), hence $\lambda\in\sigma(T_a)$.

\smallskip
Combining (1) and (2), we obtain the description of the spectrum:
$\sigma(T_a)=\overline{\{a_n\}}$.
\end{proof}

\subsection{Spectral inclusions for $p\neq 2$}

For $p\neq2$ we do not claim complete coincidence of the spectrum with the closure of the diagonal,
since for operators on Banach spaces the "diagonal form" with respect to a Schauder basis
by itself does not guarantee spectral completeness (see discussions of general spectral theory in~\cite{Conway,DunfordSchwartz}).

Nevertheless, one inclusion is stable and sufficiently general.

\begin{theorem}[Lower estimate of the spectrum]
\label{thm:incl}
Let $1<p<\infty$ and $T_a:\Lp\to\Lp$ be a bounded Walsh multiplier.
Then
\[
\overline{\{a_n:\,n\ge 0\}}\subset \sigma(T_a).
\]
\end{theorem}

\begin{proof}
Let $\lambda\in\mathbb{C}$ be such that $\lambda\notin \overline{\{a_n:\,n\ge0\}}$.
Define
\[
\delta:=\operatorname{dist}\!\bigl(\lambda,\overline{\{a_n:\,n\ge0\}}\bigr).
\]
Since $\lambda\notin \overline{\{a_n:\,n\ge0\}}$, we have $\delta>0$, and hence
\[
|a_n-\lambda|\ge \delta \qquad \text{for all }n\ge0.
\]

Define the sequence
\[
b_n:=\frac{1}{a_n-\lambda},\qquad n\ge0.
\]
Then $b\in \ell^\infty$ and $\|b\|_{\ell^\infty}\le \delta^{-1}$.
By Lemma~\ref{lem:linftyMultiplier} the operator $T_b$ is bounded on $\Lp$.

We show that $T_b$ is the inverse of $T_a-\lambda I$.
It suffices to verify this on the dense subspace
$\mathcal{P}:=\mathrm{span}\{W_n:\,n\ge0\}$ of finite linear combinations of Walsh functions.
Let
\[
f=\sum_{n=0}^N c_n W_n\in\mathcal{P}.
\]
Then, using the diagonality of multipliers on the basis $\{W_n\}$, we obtain
\[
(T_a-\lambda I)f=\sum_{n=0}^N (a_n-\lambda)c_n W_n,
\]
and consequently,
\[
T_b(T_a-\lambda I)f
=\sum_{n=0}^N b_n(a_n-\lambda)c_n W_n
=\sum_{n=0}^N c_n W_n=f.
\]
Similarly,
\[
(T_a-\lambda I)T_b f
=\sum_{n=0}^N (a_n-\lambda)b_n c_n W_n
=\sum_{n=0}^N c_n W_n=f.
\]
Thus, on $\mathcal{P}$ we have
\[
T_b(T_a-\lambda I)=I=(T_a-\lambda I)T_b.
\]
Since $\mathcal{P}$ is dense in $\Lp$, and all operators are bounded on $\Lp$,
these equalities extend by continuity to the whole $\Lp$.
Therefore, $T_a-\lambda I$ is invertible in $\mathcal{B}(\Lp)$ and
\[
(T_a-\lambda I)^{-1}=T_b,
\]
i.e., $\lambda\in\rho(T_a)$.

Therefore $\lambda\in\rho(T_a)$, i.e. $\mathbb{C}\setminus \overline{\{a_n\}}\subset \rho(T_a)$,
which is equivalent to $\overline{\{a_n\}}\subset \sigma(T_a)$.

\end{proof}

\begin{remark}[On completeness of the spectrum description for $p\neq2$]
For $p=2$ the inclusion of Theorem~\ref{thm:incl} becomes an equality
$\sigma(T_a)=\overline{\{a_n\}}$ (Theorem~\ref{thm:spectrumL2}).
For $p\neq2$ the formal "diagonality" with respect to a Schauder basis by itself
does not guarantee that the spectrum is exhausted by the closure of the diagonal without additional structural assumptions
(see discussions of diagonal operators in Banach spaces and the role of the basis in~\cite{Conway,DunfordSchwartz,LT,Wojtaszczyk}).
In the present work we therefore fix only the strictly provable inclusion.
\end{remark}

\subsection{Compact multipliers and spectrum}

\begin{proposition}
\label{prop:compactSpectrum}
Let $1<p<\infty$ and $T_a:\Lp\to\Lp$ be a compact Walsh multiplier.
Then $0\in\sigma(T_a)$, and the set of nonzero spectral values is (possibly empty) finite or countable,
each nonzero spectral value being an eigenvalue of finite multiplicity,
and the only possible accumulation point of the nonzero spectrum is $0$. 
\end{proposition}

\begin{proof}
This is a standard property of compact operators on infinite-dimensional Banach spaces;
see, e.g.,~\cite{Conway,DunfordSchwartz}.
Since $\Lp$ is infinite-dimensional and $T_a$ is compact, $0\in\sigma(T_a)$.
The remaining statements follow from the general Riesz–Schauder theorem on the spectrum of a compact operator.
\end{proof}

\begin{remark}
Proposition~\ref{prop:point} shows that each $a_n$ is an eigenvalue (provided $T_a$ is defined on $\Lp$).
Combined with compactness (Theorem~\ref{thm:compactLp}) this agrees with the fact
that the only possible accumulation point of the set $\{a_n\}$ is $0$.
\end{remark}

\section{Examples, discussion, and boundary cases}
\label{sec:examples}

\subsection{Simple examples of symbols}

\begin{example}[compact multiplier]
Let $a_n=\frac{1}{n+1}$.
Then $a_n\to0$.
If $T_a$ is bounded in the considered regime (e.g., $1<p<2$ under the conditions from~\cite{BoundednessWalsh}),
then by Theorem~\ref{thm:compactLp} the operator $T_a$ is compact in $\Lp$.
\end{example}

\begin{example}[non-compact multiplier]
Let $a_n=(-1)^n$.
Then $a_n$ does not tend to zero, hence (provided boundedness of $T_a$) the operator $T_a$ is not compact in $\Lp$.
Similarly, in the $L^p\to L^2$ regime for $p>2$ non-compactness follows from Theorem~\ref{thm:compactLpL2}.
\end{example}

\subsection{Comparison with the trigonometric case (context)}

For trigonometric multipliers in $L^p(\mathbb{T})$ a crucial role is played by
Marcinkiewicz-type criteria and their refinements/limitations,
as well as the rich phenomenology of phase multipliers (see~\cite{Zygmund1959,SteinHA,GrafakosCF}).
In the dyadic context, Walsh multipliers often allow one to formulate results
in a more "rigid" diagonal form thanks to the structure of the basis and peculiarities of the Walsh transform
(see \cite{GolubovEfimovSkvortsov1991,SchippWadeSimon1990} and the boundedness survey in~\cite{BoundednessWalsh}).

\subsection{Boundary cases $p=1$ and $p=\infty$}

The present article deals with $1<p<\infty$.
The boundary cases $p=1$ and $p=\infty$ require separate tools.
In particular, even the definition of a "good" class of multipliers in $L^1$
is often naturally formulated in terms of the weak Lorentz spaces $L^{1,\infty}$
or through conditions on the kernel in dyadic convolution.
Similarly, for $L^\infty$ convenient criteria are given via the convolution kernel and membership in $L^1$
(cf. general facts of harmonic analysis on compact groups; see~\cite{SteinHA}).
We leave these questions outside the scope of the article, but note that in the dyadic case they are especially rich,
and a correct theory of compactness on the boundaries requires a separate investigation.

\section{Conclusion}

We have obtained strict compactness criteria for multiplier operators in the Walsh system
in the spaces $L^p[0,1]$.
In the $L^p\to L^p$ regime for $1<p<\infty$ and in the $L^p\to L^2$ regime for $2<p<\infty$, compactness is equivalent to the condition $a_n\to0$.
The proofs are based on approximation by finite-dimensional truncations and precise tail estimates,
where the Hausdorff–Young inequality and the synthesis theorem for Walsh play a key role.

In the spectral part we gave a complete description of the spectrum in the Hilbert space case $p=2$:
$\sigma(T_a)=\overline{\{a_n\}}$,
while for $p\neq2$ we formulated and proved reliable inclusions that do not rely
on implicit and potentially incorrect transfers of "diagonality" beyond $L^2$.
Further research may include clarifying conditions under which equality of the spectrum and the closure of the diagonal holds in $L^p$, $p\neq2$,
as well as analysis of quasi-spectral effects and boundary regimes $p=1,\infty$.


\end{document}